\documentclass[12pt]{article}
\usepackage{amsmath,amsfonts, amssymb, graphicx}
\title{Bounded Depth Ascending HNN Extensions and $\pi_1$-Semistability at $\infty$}
\author{Michael Mihalik }
\newtheorem{theorem}{Theorem}[section]
\newtheorem{proposition}[theorem]{Proposition}
\newtheorem{lemma}[theorem]{Lemma}

\newcounter{remarknum}
\newenvironment{remark}{\addvspace{12pt}\refstepcounter{remarknum}
\noindent{\bf Remark \arabic{remarknum}.}}{\par\addvspace{12pt}}
\newenvironment{proof}{\addvspace{12pt}\noindent{\bf Proof:}}{
$\Box$\par\addvspace{12pt}}
\newcounter{examplenum}
\newenvironment{example}{\addvspace{12pt}\refstepcounter{examplenum}
\noindent{\bf Example \arabic{examplenum}.}}{\par\addvspace{12pt}}

\begin{document}
\maketitle
\begin{abstract}
A 1-ended finitely presented group has semistable fundamental group at $\infty$ if it acts geometrically on some (equivalently any) simply connected and locally finite complex $X$ with the property that any two proper rays in $X$ are properly homotopic. If $G$ has semistable fundamental group at $\infty$ then one can unambiguously define the fundamental group at $\infty$ for $G$. The problem, asking if all finitely presented groups have semistable fundamental group at $\infty$ has been studied for over 40 years. If $G$ is an ascending HNN extension of a finitely presented group then indeed, $G$ has semistable fundamental group at $\infty$, but since the early 1980's it has been suggested that the finitely presented groups that are ascending HNN extensions of {\it finitely generated} groups may include a group with non-semistable fundamental group at $\infty$. Ascending HNN extensions naturally break into two classes, those with bounded depth and those with unbounded depth. Our main theorem shows that bounded depth finitely presented ascending HNN extensions of finitely generated groups have semistable fundamental group at $\infty$. Semistability is equivalent to two weaker asymptotic conditions on the group holding simultaneously. We show one of these conditions holds for all ascending HNN extensions, regardless of depth.  We give a technique for constructing ascending HNN extensions with unbounded depth. This work focuses attention on a class of groups that may contain a group with non-semistable fundamental group at $\infty$. 
\end{abstract}

\section {Introduction}\label{Intro}

If $H$ is a group, and $\phi:H\to H$ is a monomorphism, then the notation $\langle t,H:t^{-1}ht=\phi(h)\rangle$ stands for a presentation of a group $G$ with generators $\{t\}\cup H$  and relation set $\{t^{-1} ht=\phi( h)\hbox{ for all } h\in  H\}$ union all relations for $H$. The group $G$ is usually denoted $H\ast_{\phi}$ and called an {\it ascending HNN extension} with {\it base} $H$ and {\it stable letter} $t$. By Britton's lemma the obvious map of $H$ into $G$ is an isomorphism onto its image. 
If $F(\mathcal A)$ is the free group on the set $\mathcal A$,  $\phi:\mathcal A\to F(\mathcal A)$ is a function and $\mathcal R$ is a set of $\mathcal A$-words, then the group $G$ with presentation 
$$\mathcal P=\langle t,\mathcal A:\mathcal R,t^{-1}at=\phi(a) \hbox { for all } a\in \mathcal A\rangle$$
is an ascending HNN extension of $A$,  the subgroup of $G$ generated by $\mathcal A$. It is important to note that $\langle \mathcal A:\mathcal R\rangle$ need not be a presentation for $A$. For each integer $n>0$ and $r\in \mathcal R$, $\phi^n(r)$ may not be in the normal closure of $\mathcal R$ in $F(\mathcal A)$, but certainly $\phi^n(r)$ is a relator of $A$. In fact, when $\mathcal A$ is finite,  one would rarely expect $A$ to be finitely presented. The relations  $t^{-1}at=\phi(a)$ are called {\it conjugation relations}. 

Semistability of the fundamental group at $\infty$ for a finitely presented group is a geometric notion defined in $\S$\ref{ss}. If a finitely presented 1-ended group $G$ has semistable fundamental group at $\infty$ then the fundamental group at $\infty$ of $G$ is independent of base ray. It is unknown if all finitely presented groups are semistable at $\infty$.
To date, the strongest result in the theory of semistability and simple connectivity at  $\infty$ for ascending HNN extensions is the following:

\begin{theorem} {\bf (M. Mihalik \cite{HNN1})} \label{MM}
 Suppose $H$ is a finitely presented group $\phi:H\to H$ is a monomorphism and $G=\langle t,H:t^{-1}ht=\phi(h)\rangle$ is the resulting HNN extension. Then $G$ is 1-ended and semistable at $\infty$. If additionally, $H$ is 1-ended, then $G$ is simply connected at $\infty$. 
\end{theorem}
The line of proof used for this result fails when $H$ is only finitely generated and it has been suggested since the 1980's that a promising place to search for a group with non-semistable fundamental group at $\infty$ is among the finitely presented ascending HNN extensions with finitely generated base. More specifically, A. Ol'shanskii and M. Sapir \cite{OS1} and \cite{OS2} have constructed a finitely generated infinite torsion group $\bar {\mathcal H}$ and finitely presented ascending HNN extension $\mathcal G$ of $\bar{\mathcal H}$  which has been suggested as a possible group with non-semistable fundamental group at $\infty$. 

In $\S$\ref{HNNcomb}, we show that the collection of finitely presented ascending HNN extensions of finitely generated groups is naturally divided into two classes - those with what is called {\it bounded depth} and those of {\it infinite/unbounded depth}. If the finitely generated base is finitely presented, then the resulting ascending HNN extension has bounded depth. The Ol'shanskii-Sapir group $\mathcal G$ has bounded depth and is semistable at $\infty$ by our main theorem. 

\begin{theorem} \label{mainbd} 
Suppose $G$ is a finitely presented ascending HNN extension of a finitely generated group $A$ and $G$ has bounded depth. Then $G$ has semistable fundamental group at $\infty$.
\end{theorem}

Semistable fundamental group at $\infty$ for {\it finitely generated} groups was defined in the mid-1980's (\cite{M4}). While we are not concerned with that notion here, the following result (Theorem 4, \cite{M4}) is connected to the ideas in this paper.

\begin{theorem} Suppose $G$ is an ascending HNN extension of a finitely generated 1-ended group $A$. If $A$ is semistable at $\infty$, then $G$ is semistable at $\infty$. 
\end{theorem}

To prove Theorem \ref{mainbd} we use the main theorem of \cite{GGM1} which implies that a finitely presented group $G$ has semistable fundamental group at $\infty$ if and only if two (somewhat orthogonal) weaker semistability conditions hold for $G$.  The rest of the paper is organized as follows. 

In $\S$\ref{ss}, we define semistability at $\infty$ for spaces and groups, and list a number of equivalent formulations of this notion. Two weaker notions, the semistablility of a finitely generated subgroup $J$ in an over group $G$ and, the co-semistability of $J$ in $G$ are defined. 

In $\S$\ref{basess} we prove that if $A$ is an infinite finitely generated base group of a finitely presented ascending HNN extension $G$ and $t$ is the stable letter, then for any $N\geq 0$, $t^NAt^{-N}$  is semistable at $\infty$ in $G$ (regardless of depth).  By the main theorem of \cite{GGM1} this reduces the proof of our main theorem to showing that $G$ satisfies the second semistability condition of \cite{GGM1}. 

In $\S$\ref{HNNcomb} we review the combinatorial group theory of ascending HNN groups and define what it means for such a group to have bounded depth. Examples of Grigorchuk and Ol'shanskii-Sapir of ascending HNN extensions with bounded depth are reviewed and a method for constructing ascending HNN extensions with unbounded depth is given. 

In $\S$\ref{Smain} the bulk of the proof of our main theorem is given. We show that if $G$ is an ascending HNN extension of a finitely generated group $A$, $\mathcal P$ is a finite HNN presentation with bounded depth for $G$, and $X$ is the Cayley 2-complex for $\mathcal P$, then for each compact subset $C$ of $X$, there is an integer $N(C)\geq 0$ such that $t^NAt^{-N}$ is co-semistable at $\infty$ in $X$ with respect to $C$. We also prove a result (Theorem \ref{FP}) that considers the case when $A$ is finitely presented and connects this case to several papers already in the literature. When $A$ is finitely presented and $C$  is compact in $X$,  we show there is an integer $N(C)\geq 0$ and compact set $Q(C)$ containing $C$ such that loops in $X-(t^{N}At^{-N}) Q$ are homotopically trivial in $X-(t^{N}At^{-N}) Q$.

\section{The basics of semistability at $\infty$ for groups}\label{ss} 

Suppose $K$ is a locally finite connected CW complex. A {\it ray} in $X$ is a map $r:[0,\infty)\to K$. 
The space $K$ has {\it semistable fundamental group at $\infty$} if any two proper rays in $K$ converging to the same end are properly homotopic.  Suppose  $C_0, C_1,\ldots $ is a collection of compact subsets of a 1-ended locally finite complex $K$ such that $C_i$ is a subset of the interior of $C_{i+1}$ and $\cup_{i=0}^\infty C_i=K$, and $r:[0,\infty)\to K$ is proper, then $\pi_1^\infty (K,r)$ is the inverse limit of the inverse system of groups:
$$\pi_1(K-C_0,r)\leftarrow \pi_1(K-C_1,r)\leftarrow \cdots$$
This inverse system is pro-isomorphic to an inverse system of groups with epimorphic bonding maps if and only if $K$ has semistable fundamental group at $\infty$.  When $K$ is 1-ended with semistable fundamental group at $\infty$, $\pi_1^\infty (K,r)$ is independent of proper base ray $r$. 

If for any compact set $C$ in $K$ there is a compact set $D$ in $K$ such that loops in $K-D$ are homotopically trivial in $X-C$ (equivalently the above inverse sequence of groups is pro-trivial), then $K$ is {\it simply connected at $\infty$}.

There are a number of equivalent forms of semistability which are collected as Theorem 3.2 of \cite{CM2}.

\begin{theorem}\label{ssequiv} {\bf (G. Conner and M. Mihalik \cite{CM2})}
Suppose $K$ is a locally finite, connected and 1-ended CW-complex. Then the following are equivalent:
\begin{enumerate}
\item $K$ has semistable fundamental group at $\infty$.
\item For any proper ray $r:[0,\infty )\to K$ and compact set $C$, there is a compact set $D$ such that for any third compact set $E$ and loop $\alpha$ based on $r$ and with image in $K-D$, $\alpha$ is homotopic $rel\{r\}$ to a loop in $K-E$, by a homotopy with image in $K-C$. 
\item For any compact set $C$ there is a compact set $D$ such that if $r$ and $s$ are proper rays based at $v$ and with image in $K-D$, then $r$ and $s$ are properly homotopic $rel\{v\}$, by a proper homotopy in $K-C$. 
\end{enumerate}
If $K$ is simply connected, then a fourth equivalent condition can be added to this list:

4. Proper rays $r$ and $s$ based at $v$ are properly homotopic $rel\{v\}$. 
\end{theorem}

If $G$ is a finitely presented group and $Y$ is  a finite complex with $\pi_1(Y)=G$ then $G$ has {\it semistable (respectively simply connected)} fundamental group at $\infty$ if the universal cover of $Y$ has semistable (respectively simply connected) fundamental group at $\infty$. This definition only depends on the group $G$.

In \cite{GGM1} we consider finitely generated groups acting (perhaps not co-compactly) as covering transformations on 1-ended CW complexes $X$ and we say what it means for such a group to be semistable at $\infty$ in $X$ with respect to a given compact subset of $X$. In this paper we only need consider a more simple notion. Suppose $A$ is a finitely generated infinite subgroup of a finitely presented 1-ended group $G$. Say  $\mathcal A\cup \mathcal S$ is a finite generating set of $G$, where $\mathcal A$ generates $A$. Let $X$ be the Cayley 2-complex for some finite presentation $\mathcal P$ (with generating set $\mathcal A\cup \mathcal S$) of $G$. So $X$ is the simply connected  2-dimensional complex with 1-skeleton equal to the Cayley graph of $G$ with respect to $\mathcal A\cup \mathcal S$. The vertex set of $X$ is $G$ and each edge of $X$ is labeled by an element of $\mathcal A\cup \mathcal S$. For each vertex $v$ of $X$ and relation $r$ of $\mathcal P$ there is a 2-cell with boundary equal to the edge path loop at $v$ with edge labels spelling the word $r$.   Let $\ast$ be the identity vertex of $X$. Let $\Lambda(A, \mathcal A)\subset X$ be the Cayley graph of $A$ with respect to $\mathcal A$. If $g\in G$ and $q$ is an edge path in $g\Lambda$, then $q$ is called an $\mathcal A $-{\it path} in $X$. Note that $q$ is an $\mathcal A$-path if and only if each edge of $q$ is labeled by an element of $\mathcal A$. 

If $g\in G$ and $C$ is compact in $X$ then we say $gAg^{-1}$ is {\it  semistable at $\infty$ in} $X$ (or in $G$) {\it with respect to $C$} if there is a compact set $D(C)\subset X$ such that if $r$ and $s$ are two proper edge path rays in $g\Lambda(A,\mathcal A)-D$ 
based at the same vertex $v\in gA$ then $r$ and $s$ are properly homotopic $rel\{v\}$ by a proper homotopy in $X-C$. This definition is equivalent to the one of \cite{GGM1}. If $gAg^{-1}$ is semistable at $\infty$ with respect to every compact subset of $X$, then we say $gAg^{-1}$ is semistable at $\infty$ in $X$ (or in $G$). If $A$ is 1-ended and semistable at $\infty$, then $gAg^{-1}$ is always semistable at $\infty$ in $X$ ($G$).

In $\S$\ref{basess} we prove:

\begin{proposition}\label{strongss} 
If $G$ is a finitely presented ascending HNN extension of a finitely generated infinite group $A$ and $t$ is the stable letter, then for all $N\geq 0$, $t^NAt^{-N}$ is  semistable at $\infty$ in $G$.
\end{proposition}

The main theorem of \cite{GGM1} is significantly more general than Theorem \ref{GGM}.  In \cite{GGM1}, the main result does not require an overgroup $G$ acting cocompactly on $Y$, only that $Y$ be 1-ended and for each compact subset $C$ of $Y$, the existence a finitely generated group $J$ acting as covering transformations on $Y$ and satisfying conditions 1) and 2) below. The notion of a group $J$ being co-semistable at $\infty$ in a space is a bit technical and we define this afterwards.

\begin{theorem} \label{GGM} {\bf (R. Geoghegan, C. Guilbault and M. Mihalik \cite{GGM1})}
Suppose $G$ is a 1-ended finitely presented group acting cocompactly on a simply connected locally finite CW-complex $Y$. If for each compact set $C\subset Y$ there is an infinite finitely generated subgroup $J$ of $G$ such that 

1) $J$ is  semistable at $\infty$ in $Y$ with respect to $C$ and 

2) $J$ is co-semistable at $\infty$ in $Y$ with respect to $C$, 

\noindent then $Y$ (and hence $G$) has semistable fundamental group at $\infty$.  
\end{theorem}

The converse of Theorem \ref{GGM} is rather straightforward.  In fact, if $Y$ (equivalently $G$) has semistable fundamental group at $\infty$, then suppose $C$ is any compact subset of $Y$ and $J$ is any infinite finitely generated subgroup of $G$ then conditions 1) and 2) hold for $J$ and $C$. Interestingly, our proof of the main theorem of this paper relies on selecting different groups $J$ for different compact sets $C$ satisfying 1) and 2).  
We apply Theorem \ref{GGM} when $G$ is an ascending HNN extension of a finitely generated group $A$, and $G$ acts cocompactly on $Y$ the Cayley 2-complex of $G$ with respect to some finite HNN presentation $\mathcal P$ (see $\S$\ref{Intro}).  In our situation, all of the subgroups $J$ of Theorem \ref{GGM} will have the form $t^NAt^{-N}$ for some $N\geq 0$. Proposition \ref{strongss} resolves part 1) of Theorem \ref{GGM} for all compacts sets. All that remains to be shown is that for each compact set $C$ in $X$ there is an integer $N(C)\geq 0$ such that $t^NAt^{-N}$ is co-semistable at $\infty$ in $Y$ with respect to $C$. We now define what that means.

Suppose $J$ is an infinite finitely generated group acting as covering transformations on the 1-ended, simply connected and locally finite CW-complex $Y$. A subset $S$ of $Y$ is {\it bounded} in $Y$ if $S$ is contained in a compact subset of $Y$. Otherwise $S$ is {\it unbounded} in $Y$. Let $q:Y\to J\backslash Y$ be the quotient map.
If $K$ is a subset of $Y$, and there is a compact subset $C_1$ of $Y$ such that $K\subset JC_1$ (equivalently $q(K)$ has image in a compact set), then $K$ is a $J$-{\it bounded} subset of $Y$. Otherwise $K$ is a $J$-{\it unbounded} subset of $Y$. If $r:[0,\infty)\to Y$ is a proper edge path ray and $qr$ has image in a compact subset of $J\backslash Y$ then 
$r$ is said to be $J$-{\it bounded}. 
Equivalently, $r$ is a $J$-bounded proper edge path ray in $S$ if and only if $r$ has image in $J C_1$ for some compact set $C_1\subset Y$. Let $\ast$ be a base vertex in $Y$. When $r$ is $J$-bounded there is an integer $M$ (depending only on $C_1$ and fixed terms) such that each vertex of $r$ is (using edge path distance) within $M$ of a vertex of $J \ast\subset Y$.

We say $J$ is {\it co-semistable at $\infty$ in $Y$ with respect to the compact subset $C$ of $Y$} if there is a compact subcomplex $C_1$ of $Y$ such that for each $J$-unbounded component $U$ of $Y-(JC_1)$, and any $J$-bounded proper ray $r$ in $U$ ``loops in $U$ and based on $r$ can be properly pushed to infinity along $r$, avoiding $C$". More specifically:



 For any loop $\alpha:[0,1]\to U$ with $\alpha(0)=\alpha(1)=r(0)$ there is a proper homotopy $H:[0,1]\times [0,\infty)\to Y-C$ such that $H(t,0)=\alpha(t)$ for all $t\in [0,1]$ and $H(0,s)=H(1,s)=r(s)$ for all $s\in [0,\infty)$.

\section{Base group semistability in an ascending HNN extension} \label{basess}
 
In this section we prove three lemmas that imply Proposition \ref{strongss}. This shows that an infinite finitely generated base group is always semistable at $\infty$ in an ascending HNN extension (regardless of bounded or unbounded depth).
Begin with a finite presentation for a group $G$ which is an ascending HNN extension with base group a finitely generated group $A$ with finite set of generators $\mathcal A$:
$$\mathcal P=\langle t, \mathcal A: \mathcal R, t^{-1}at=\phi(a)\hbox{ for all }a\in \mathcal A\rangle$$

Here $\mathcal R$ is a finite subset of the free group $F(\mathcal A)$. Consider the homomorphism $P_0:G\to \mathbb Z$ that kills the normal closure of $A$. If $g\in G$ and $P_0(g)=N$, we say $g$ is in {\it level} $N$. Let $X$ be the Cayley 2-complex for the presentation $\mathcal P$ of $G$.
Then $P_0$ can be extended to $P:X\to \mathbb R$ by taking each 2-cell corresponding to an element of $\mathcal R$ to $P_0(v)$ for any vertex $v$ of the cell, and if $D$ is  2-cell corresponding to a conjugation relation $t^{-1}at=\phi(a)$ for $a\in \mathcal A$, then $P$ maps $D$ to the interval $[N,N+1]$ (where the edge of $D$ corresponding to $a\in \mathcal A$ is mapped by $P_0$ to $N$ and those corresponding to $\phi(a)$ are mapped to $N+1$), in the obvious way.  

\begin{lemma}\label{push} 
Let $e:[0,1]\to X$ be an edge in $X$ with $e(0)=v$, $e(1)=w$ and label $a\in \mathcal A$. Let $r_v$ and $r_w$ be the edge path rays at $v$ and $w$ (respectively) each of whose edges is labeled $t$. There is a proper map $H_e:[0,1]\times [0,\infty)\to X$ such that $H_e(t,0)=e(t)$, $H_e(0,t)=r_v(t))$, $H_e(1,t)=r_w(t)$ and $P(H_e([0,1]\times [N,N+1]))\subset [N,N+1]$.
\end{lemma}

\begin{proof}
On $[0,1]\times [0,1]$ define $H_e$ to have image the 2-cell at $v$ with boundary label $at\phi(a^{-1})t^{-1}$. Iterate to define $H_e$ as in Figure 1. Note that if $\phi(a)$ has length $L$ then the image of $H_e$ on $[0,1]\times [1,2]$, consists of $L$ conjugation relation 2-cells (each of which is mapped by $P$ to $[1,2]$). 

\vspace {.5in}
\vbox to 2in{\vspace {-2in} \hspace {-.3in}
\hspace{-1 in}
\includegraphics[scale=1]{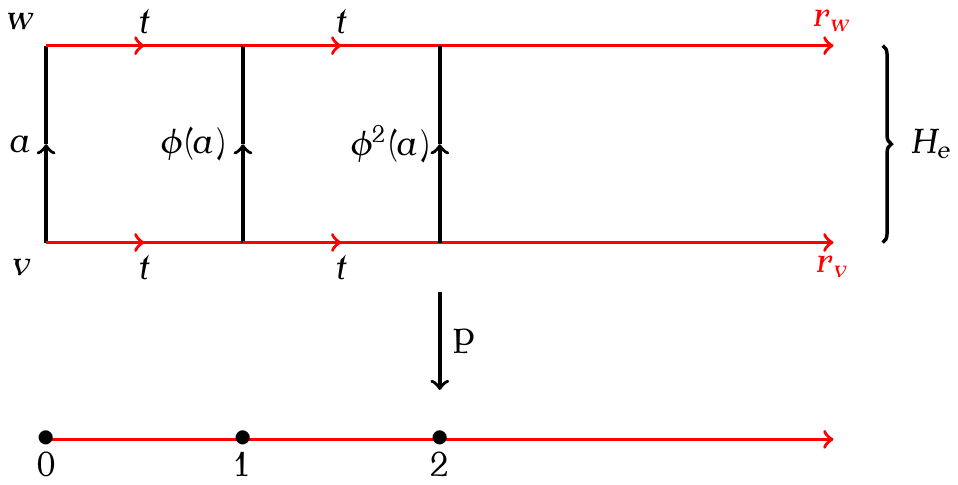}
\vss }


\centerline{Figure 1}

\medskip
To see that $H_e$ is proper, let $C$ be compact in $X$. Then $P(C)\subset [-N,N]$ for some integer $N\geq 0$. But then $H_e^{-1} (C)\subset [0,1]\times [0,N]$. 
\end{proof}

 Recall $\Lambda$ is the Cayley graph of $A$ with respect to $\mathcal A$ and we assume $\ast\in \Lambda\subset X$ where $\ast$ is the identity vertex.
\begin{lemma} \label{straight} 
Suppose $C$ is compact in $X$. There are only finitely many $\mathcal A$-edges $e$ in $\Lambda$ such that the image of $H_e$ (see Lemma \ref{push}) intersects $C$. 
\end{lemma}
\begin{proof}
 If $v\in A$, let $r_v$ be the proper edge path ray at $v$, each of whose edges is labeled $t$. If $e$ is an edge of $\Lambda$ with initial point $v$, let $H_e$ be the proper homotopy of Lemma \ref{push}. For any integers $S>R\geq 0$, $P(H_e([0,1]\times [R,S]))\subset [R,S]$. Say that $P(C)\subset [-N,N]$ for $N\geq 0$. Then for any edge $e$ of $\Lambda$, $$H_e([0,1]\times [N+1,\infty))\cap C=\emptyset$$ 
(since $P(C)\subset [-N,N]$ and $P H_e([0,1]\times [N+1,\infty))\subset [N+1,\infty)$). Let  $L$ be the length of the longest word in $\{\phi(a_1),\ldots, \phi(a_n)\}$. So for any integer $K\geq 0$, the length of the $\mathcal A$-word $H_e([0,1]\times \{K\}$ is $\leq L^K$ (if $e$ has label $a\in\mathcal A$, then $H_e([0,1]\times \{K\}$ has label $\phi^K(a)$).  For any edge $e$ of $\Lambda$ with initial vertex $v$, 
$$H_e([0,1]\times [0,N])\subset St^{L^N+N}(v).$$ 
There are only finitely many vertices $v$ of $\Lambda$ such that $St^{L^N+L}(v)\cap C\ne \emptyset$ and so there are only finitely many edges $e$ of $\Lambda$ such that the image of $H_e$ intersects $C$. 
\end{proof}

\begin{lemma}\label{string}
Suppose $s=(s_0,s_1,\ldots )$ is a proper edge path ray in $\Lambda\subset X$. If $v$ is the initial point of $s$ let  $r_v$ be the edge path at $v$ each of whose edges is labeled $t$, then there is a proper homotopy $H_s:[0,\infty)\times [0,\infty)\to X$ of $s$ to $r_v$ $rel\{v\}$ defined so that $H_s$ restricted to $[N,N+1]\times [0,\infty)$ is $H_{s_N}$ (i.e. $H_s(N+x,y)=H_{s_N}(x,y)$ for all $(x,y)\in [0,1]\times [0,\infty)$).
\end{lemma}
\begin{proof}
Since $H(0,y)=r_v(y)$ and $H(x,0)=s$, $H$ is a homotopy of $r_v$ to $s$ $rel\{v\}$. It remains to show that $H$ is proper. If $C$ is compact in $X$, then by Lemma \ref{straight} there are only finitely many edges $e$ of $s$ such that the image of $H_e$ intersects $C$. Choose $N$ such that for all $n>N$, $H_{s_n}$ avoids $C$. Then $H_s^{-1}(C)=\cup_{i=1}^NH_{s_i}^{-1}(C)$. This last set is a finite union of compact sets since each $H_{s_i}$ is proper. 
\end{proof}

\begin{proof} {\bf (of Proposition \ref{strongss})}
We show that for any integer $N\geq 0$, the group $t^NAt^{-N}$ is semistable at $\infty$ in $X$ ($G$). Let $C$ be compact in $X$. If $v\in A$, let $r_v$ be the proper edge path ray at $v$, each of whose edges is labeled $t$. If $e$ is an edge of $\Lambda$ with initial point $v$, let $H_e$ be the proper homotopy of Lemma \ref{push}. By Lemma \ref{straight} there are only finitely many edges $e$ of $\Lambda$ such that the image of $H_e$ intersects $t^{-N}C$.  Choose $D$ compact such that $D$ contains $t^{-N}C$ and all of these edges. If $s$ and $s'$ are proper $\mathcal A$-rays at $v\in \Lambda-D$ then the proper homotopies $H_s$ and $H_{s'}$ of Lemma \ref{string} both avoid $t^{-N}C$ so that both $s$ and $s'$ are properly homotopic $rel \{v\}$ to $r_v$ by homotopies in $X-t^{-N}C$. Combining $H_s$ and $H_{s'}$ we have $s$ is properly homotopic $rel\{v\}$ to $s'$ by a homotopy $H$ in $X-t^{-N}C$.  Now $t^NH$ is a proper homotopy $rel\{t^Nv\}$ of $t^Ns$ to $t^Ns'$ in $X-C$ and $t^NAt^{-N}$ is semistable at $\infty$ in $X$.
\end{proof}

\section{Ascending HNN extension combinatorics}\label{HNNcomb} 

Suppose $\mathcal A$ is a finite set, $\phi:F(\mathcal A)\to F(\mathcal A)$ is a homomorphism of the free group, $\mathcal R$ is a finite set of words in $F(\mathcal A)$ and $G$ is the (finitely presented) ascending HNN extension with the following HNN presentation:
$$(\ast) \ \ \ \ \ \ \ \ \ \ \ \ \ \mathcal P=\langle t, \mathcal A: \mathcal R, t^{-1}at=\phi(a)\hbox{ for all }a\in \mathcal A\rangle$$
The base group of this HNN extension is $A$, the subgroup of $G$ generated by $\mathcal A$. In this paper, we are only interested in the case when $\mathcal A$ is finite. In order to define what it means for an ascending HNN extension to have bounded depth, we must first understand $ker(p)$ where $p$ is the homomorphism $p:F(\mathcal A)\to A$ (defined by $p(a)=a$ for $a\in \mathcal A$).  Certainly $ker(p)$ contains $N_0(\mathcal R,\phi)\equiv N( \cup _{i=0}^\infty \phi^i (\mathcal R))$, where $N( \cup _{i=0}^\infty \phi^i (\mathcal R))$ is the normal closure of $ \cup _{i=0}^\infty \phi^i (\mathcal R)$ in $F(\mathcal A)$. But it may be that for some word 
$w\in F(\mathcal A)$ and some integer $m$, $\phi^m(w)\in N_0(\mathcal R,\phi)$, and $w\not \in N_0(\mathcal R,\phi)$. Then $w\in ker(p)$. Consider the normal subgroup of $F(\mathcal A)$: 
$$N^{\infty}(\mathcal R,\phi)\equiv \cup _{i=0}^{\infty} \phi^{-i}(N_0(\mathcal R,\phi))\triangleleft F(\mathcal A).$$
It is well known to experts that $\phi^{-i}(N_0(\mathcal R,\phi))< \phi^{-i-1}(N_0(\mathcal R,\phi))$ (see theorem \ref{rel}) so that  $N^{\infty}(\mathcal R,\phi)$ is an ascending union of normal subgroups of $F(\mathcal A)$ and that $N^{\infty}(\mathcal R,\phi)$ is the kernel of $p$, so 
$$A=\langle \mathcal A:N^{\infty}(\mathcal R,\phi)\rangle$$ 

If there is an integer $B$ such that $N^\infty(\mathcal R,\phi)=\cup _{i=0}^B\phi^{-i}(N_0(\mathcal R,\phi))$ then the presentation $\mathcal P$ of $G$ has {\it bounded depth}. Our main theorem shows that if $\mathcal P$ has bounded depth, then $G$ is semistable at $\infty$ (Theorem \ref{mainbd}). It is not always the case that such ascending HNN extensions have bounded depth. (See Theorem \ref{Osin}.) 

As in $\S$\ref{basess}, $P_0:G\to \mathbb Z$ is the homomorphism that kills the normal closure of $A$. If $X$ is the Cayley 2-complex for the presentation $\mathcal P$ of $G$ given in $(\ast)$ (with vertex set $G$), then $P_0$ extends to $P:X\to \mathbb R$. If $g\in G$ and $P_0(g)=N$, $g$ is in {\it level} $N$.  

\begin{remark} \label{loopkill}
An edge path loop in level $L$ of $X$, whose labeling  defines an element of $\cup _{i=0}^B\phi^{-i}(N_0(\mathcal R,\phi))$, is homotopically trivial by a combinatorial homotopy $H$ such that $P(H)$ has image in $(-\infty,L+B]$. Note that if $\alpha$ is an edge path loop in level $L$ labeled by an element of $N(\mathcal R)$ (the normal closure of $\mathcal R$ in $F(\mathcal A)$) then $\alpha$ can be killed by a homotopy in level $L$. 
If $\alpha$ has initial vertex $v$ in level $L$ and labeling $\phi(r)$ for $r\in N(\mathcal R)$, then using only conjugation relations, $\alpha$ is homotopic to an edge path loop at $v$ with labeling $(t^{-1},\beta, t)$ where $\beta$ has labeling $r$ and image in level $L-1$. Since $\beta$ is homotopically trivial in level $L-1$, the loop $\alpha$ can be killed by a homotopy $H$ such that $P(H)$ has image in $[L-1,L]$. This homotopy only uses the homotopy that kills $\beta$ in level $L-1$ and the  conjugation relation 2-cells connecting $\alpha$ and $\beta$. If $\alpha$ has label in $\phi^{-1}(N(\mathcal R))$ (so $\phi(\alpha)=r\in N(\mathcal R)$) then $\alpha$ can be killed by a homotopy $H$ such that $P(H)$ has image in $[L,L+1]$.
\end{remark}

In the case that $A$ is finitely generated and the image of $\phi:A\to A$ is of finite index in $A$, then $A$ is ``commensurated" in $G$ and $G$ is semistable at $\infty$ (see Corollary 4.9 of \cite{CM2}). 

For $\mathcal A$ finite, the group $G=\langle t, \mathcal A: \mathcal R', t^{-1}at=\phi(a) \hbox{ for } a\in \mathcal A\rangle$ (with $\mathcal R' \subset F(\mathcal A)$) is an ascending HNN extension with {\it bounded depth $D$ and root $\mathcal R$} if the kernel of the homomorphism $p:F(\mathcal A)\to A$ (defined by $p(a)=a$ for all $a\in \mathcal A$) is $\phi^{-D}(N_0(\mathcal R,\phi))\equiv \phi^{-D}(N(\cup _{i=0}^\infty \phi^i (\mathcal R)))$ for some finite set of words $\mathcal R$ in $F(\mathcal A)$. In this case, $G$ has finite presentation: $\langle t, \mathcal A:\mathcal R,t^{-1}at=\phi(a) \hbox{ for all } a
\in \mathcal A\rangle$.

\begin{example} \label{G}
R. Grigorchuk (\cite{GR1} and \cite{GR2}) constructed a finitely generated infinite torsion group $G$ of intermediate growth having solvable word problem. He also showed that $G$ was the base group of a finitely presented ascending HNN extension (which is the first example of a finitely presented cyclic extension of an infinite torsion group). I. Lys\" enok \cite{L} produced the following recursive presentation of $G$:
$$G\equiv \langle a,c,d:\sigma^n(a^2), \sigma^n((ad)^4), \sigma^n(adacac)^4), n\geq 0\rangle$$
where $\sigma (a)=aca, \sigma (c)=cd$ and $\sigma (d)=c$. It can be shown that the ascending HNN extension $E$ with presentation: 
$$\langle a,c,d,t: a^2=(ad)^4=(adacac)^4=1, t^{-1}at=aca, t^{-1}ct=dc, t^{-1}dt=c\rangle$$ 
has base group $G$ generated by $\{a,c,d\}$ and $E$ has bounded depth with root $\{a^2,c^2, d^2, (ad)^4, (adacac)^4\}$.
The group $E$ was the first example of a finitely presented amenable but not elementary amenable group. In $\S$5 of \cite{M6}, M. Mihalik shows that $E$ is simply connected at $\infty$. The notion of a finitely generated group being simply connected at $\infty$ is introduced in \cite{M6}, and the group $G$ is shown to be simply connected at $\infty$.
\end{example}

\begin{example}\label{OS} 
A. Ol'shanskii and M. Sapir \cite{OS1} and \cite{OS2} construct a finitely presented ascending HNN extension $\mathcal G$, where the base group $\bar{\mathcal H}$ is a finitely generated infinite torsion group. In contrast to Grigorchuk's group (Example \ref{G}) the base group has finite exponent, and $\mathcal G$ is not amenable (see Theorem 1.1 of \cite{OS1}).  The group $\mathcal G$ has been suggested as a possible non-semistable at $\infty$ group, but it is clear from the equations (5)-(8) in $\S$1.2 of \cite{OS1} that $\mathcal G$ has an ascending HNN presentation with depth one, and so by our main theorem is semistable at $\infty$.  We give a brief summary. A finite set of words $\mathcal R$ is determined in $F_C=\langle c_1,\ldots, c_m\rangle$ a free group of rank $m$.  A monomorphism $\phi:F_C\to F_C$ is defined and $\mathcal R'$ is defined to be $\cup_{i=1}^{\infty}\{\phi^i(r):r\in \mathcal R\}$. The base group of their ascending HNN extension has presentation
$$\bar {\mathcal H}=\langle c_1,\ldots, c_m: \mathcal R\cup \mathcal V\cup \mathcal R'\rangle$$
where $\mathcal V$ is the set of elements $u^n$  for all $u\in F_C$ (and $n$ a fixed large odd number). In particular, $\bar{\mathcal H}$ is an infinite torsion group. A finitely presented ascending HNN extension  of $\bar{\mathcal H}$ has infinite presentation
$$\mathcal G=\langle t, c_1,\ldots, c_m: t^{-1}c_it=\phi(c_i), \mathcal R\cup \mathcal R'\cup \mathcal V\rangle$$
(This follows equation (7) of \cite{OS1}.) Clearly the relations $\mathcal R'$ are a consequence of $\mathcal R$ and the conjugation relations and so can be removed. It is then argued that each relation $v^n$ of $\mathcal V$ is $\phi^{-1} (v')$ where $v'$ is a consequence of $\mathcal R$ and the conjugation relations. In particular, the above presentation of $\mathcal G$ can be reduced to the presentation 
$$\mathcal G=\langle t, c_1,\ldots, c_m: t^{-1}c_it=\phi(c_i), \mathcal R\rangle$$ 
and this presentation has depth 1. It seems unlikely that $\mathcal G$ has an ascending HNN presentation with depth 0. One must wonder if for every integer $N>0$ there are finitely presented ascending HNN groups $\mathcal G_N$ with ascending HNN presentations of depth $N$ but $\mathcal G_N$ does not have such a presentation of depth $N-1$.
\end{example}

\begin{theorem} \label{rel}
Suppose $G$ is the ascending HNN extension with finite presentation: 
$$\mathcal P=\langle t,\mathcal A: \mathcal R, t^{-1}at=\phi(a) \hbox{ for all } a\in \mathcal A\rangle$$
where $\phi:F(\mathcal A)\to F(\mathcal A)$ is a (finite rank) free group homomorphism. Then $A$, the subgroup of $G$ generated by $\mathcal A$, has presentation:
$$A=\langle \mathcal A:N^{\infty}(\mathcal R,\phi)\equiv \cup_{i=0}^\infty\phi^{-i}(N(\cup _{j=0}^\infty \phi^j(\mathcal R)))\rangle.$$ 
Furthermore, we have the relations:
 
\begin{enumerate} 
\item  $\phi^{-i}(N(\cup _{j=0}^\infty \phi^j(\mathcal R)))\subset \phi^{-(i+1)}(N(\cup _{j=0}^\infty \phi^j(\mathcal R))) 
\hbox{ for all }i\geq 0,\hbox{ and}$
\item  $\phi(N^{\infty} (\mathcal R,\phi))\subset N^\infty(\mathcal R,\phi)=\phi^{-1}(N^\infty(\mathcal R,\phi))$
\end{enumerate}
\end{theorem}
\begin{proof}
Note that 
$$\phi(N(\cup _{j=0}^\infty \phi^j(\mathcal R)))\subset N(\cup _{j=1}^\infty \phi^j(\mathcal R)))\subset N(\cup _{j=0}^\infty \phi^j(\mathcal R))\hbox{ so that}$$ 
$$ N(\cup _{j=0}^\infty \phi^j(\mathcal R)))\subset \phi^{-1}( N(\cup _{j=0}^\infty \phi^j(\mathcal R)))$$
and so relation 1) follows. 

To simplify notation, let $N^\infty=N^\infty(\mathcal R,\phi)$ and $N_i=\phi^{-i}(N(\cup _{j=0}^\infty \phi^j(\mathcal R)))$ for $i\geq 0$, so that $N^\infty=\cup_{i=0}^\infty N_i$ and by 1), $N_i\subset N_{i+1}=\phi^{-1}(N_i)$.
Suppose $a\in \phi^{-1}(N^{\infty})$. Then $\phi(a)\in N^{\infty}$ and so $\phi(a)\in N_i$ for some $i\geq 0$. Then $a\in \phi^{-1}(N_i)=N_{i+1}\subset N^{\infty}$ and we have shown that $\phi^{-1}(N^\infty)\subset N^\infty$.

Next suppose $a\in N^\infty$. Then for some $i\geq 0$, $a\in N_i$. By 1), $a\in N_{i+1}=\phi^{-1}(N_i)\subset \phi^{-1}(N^\infty)$. We have shown that $N^\infty(\mathcal R,\phi)\subset \phi^{-1}(N^\infty(\mathcal R,\phi))$. Combining we have $N^\infty= \phi^{-1}(N^\infty)$ and relation 2) follows. 

Let $A_1$ be the group with presentation $\langle \mathcal A:N^{\infty}(\mathcal R,\phi)\rangle$. To finish the theorem we must show that $A=A_1$. Let $p_1:F(\mathcal A)\to A_1$ (determined by $p_1(a)=a$ for all $a\in \mathcal A$) 
be the quotient homomorphism. By 2), the map $\phi_1:A_1\to A_1$ that extends the map $\phi_1(p_1(a))=p_1(\phi(a))$ for all $a\in \mathcal A$ is a homomorphism. This gives a commutative diagram: 
$$F(\mathcal A){\buildrel \phi\over \longrightarrow}F(\mathcal A)$$
$$\downarrow p_1\ \ \ \ \ \  \downarrow p_1$$
$$ A_1\ \ \ {\buildrel \phi_1\over \longrightarrow}\ \  A_1$$
Next we show that $\phi_1$ is a monomorphism. Suppose $w_1\in ker(\phi_1)$. Let $w\in F(\mathcal A)$ be such that $p_1(w)=w_1$. Then $p_1(\phi(w))=1$ and so $\phi(w)\in ker(p_1)=N^\infty$ and $w\in \phi^{-1}(N^\infty) = N^\infty$. 
Then $w_1=p_1(w)=1\in A_1$ and $\phi_1$ is a monomorphism. 

Consider the ascending HNN extension: 
$$A_1\ast_\phi=\langle t, \mathcal A: N^\infty(\mathcal R,\phi), t^{-1}at=\phi(a)\rangle \hbox{ for all } a\in \mathcal A$$
 with base group $A_1$. Since each relation in $N^\infty(\mathcal R,\phi)$ is a consequence of $\mathcal R$ and the conjugation relations, this group also has presentation $\mathcal P$. By Britton's lemma $A=A_1$.
\end{proof}

Suppose $G$ has finite presentation $\langle t, \mathcal A: \mathcal R, t^{-1}at=\phi(a)\hbox{ for } a\in \mathcal A\rangle$. Here $\phi:F(\mathcal A)\to F(\mathcal A)$ is a homomorphism. Let $N_0\equiv N(\cup _{j=0}^\infty \phi^j(\mathcal R))\triangleleft F(\mathcal A)$,  $N_i\equiv \phi^{-i}(N_0)$ and $A$ be the subgroup of $G$ generated by $\mathcal A$, so that $G$ is the ascending HNN extension, with base $A$ and stable letter $t$.  Let $p:F(\mathcal A)\to A$ be the homomorphism extending the map taking $a$ to $a$ for all $a\in \mathcal A$.

It seems that there is some potential to find a finitely presented group that is not semistable at $\infty$ if one could find a finitely presented ascending HNN extension $\langle t, \mathcal A: \mathcal R, t^{-1}at=\phi(a)\hbox{ for } a\in \mathcal A\rangle $, such that the ascending chain of normal subgroups $N_k$ of $F(A)$ do not stabilize. The following approach gives a general method of constructing infinite depth ascending HNN presentations. In particular, when $A_0$ is a non-Hopfian group and $\phi_0:A_0\to A_0$ is an epimorphism with non-trivial kernel, then there is a corresponding ascending HNN extension with infinite depth. 

\begin{theorem} \label{Osin} 
Suppose the group $A_0$ has finite presentation $\langle \mathcal A:\mathcal R\rangle$ and
$\phi_0:A_0\to  A_0$ is a homomorphism with non-trivial kernel $K_0$ such that the following diagram (with $F({\mathcal A})$ the free group on $\mathcal A$ and $q(a)=a$ for $a\in \mathcal A$) commutes:
$$F(\mathcal A){\buildrel \phi\over \longrightarrow}F(\mathcal A)$$
$$\ \ \downarrow q\ \ \ \ \ \  \downarrow q$$
$$\ \  A_0\ \ {\buildrel \phi_0\over \longrightarrow}\ A_0$$
If the ascending sequence $\{K_i=\phi_0^{-i}(K_0)=ker(\phi_0^{i+1})\}$ of normal subgroups of $A_0$ does not stabilize (in particular when $\phi_0$ is an epimorphism), then the group $G$ with ascending HNN presentation 
$$\mathcal P\equiv \langle t,\mathcal A:\mathcal R, t^{-1}at=\phi(a) \hbox{ for all }a\in \mathcal A\rangle$$
has unbounded depth.
\end{theorem}
\begin{proof}
First observe that if $\phi_0$ is an epimorphism, and $k\in K_0-1$, then there is $k_n$ such that $\phi_0^n(k_n)=k$. In particular, $k_n\in ker(\phi_0^{n+1})-ker(\phi_0^n)$.
Note that $ker(q)=N(\mathcal R)\triangleleft F(\mathcal A)$. If $r\in N(\mathcal R)$, then $q(\phi(r))=1$ and so $\phi(N(\mathcal R))\subset N(\mathcal R)$ and (retaining the notation of Theorem \ref{rel})
$$N_0=N(\cup_{i=0}^\infty\phi^i(\mathcal R))=N(\mathcal R)=ker(q).$$
For the subgroup $A$  of $G$ determined by $\mathcal A$ there is a commutative diagram:
$$F(\mathcal A){\buildrel \phi\over \longrightarrow}F(\mathcal A)$$
$$\downarrow p\ \ \ \ \ \  \downarrow p$$
$$ A\ \ \ {\buildrel \phi_1\over \longrightarrow}\ \  A$$
Observed that $A$ is a quotient of $A_0$ where the element $q(a)$ is mapped to $p(a)$ for all $a\in\mathcal A$ and the following diagram commutes: 
$$A_0\ \ {\buildrel \phi_0\over \longrightarrow}\ \ A_0$$
$$\ \ \ \downarrow q_0\ \ \ \ \ \ \downarrow q_0$$
$$ A\ \ \ {\buildrel \phi_1\over \longrightarrow}\ \  A$$

$(\ast)$ If  $\phi_0$ is an epimorphism, then since $q_0$ is an epimorphism $\phi_1$ is also an epimorphism. In any case, $G =A\ast_{\phi_1}$ and when $\phi_0$ is an epimorphism, $\phi_1$ is an isomorphism.


Let $N_i=\phi^{-i}(N_0)\triangleleft F(\mathcal A)$. By Theorem \ref{rel}.1 $N_{i-1}\leq N_i$. For $i>0$ we show $N_i\ne N_{i-1}$ when $K_i\ne K_{i-1}$, so that $\mathcal P$ has unbounded depth when $\{K_i\}$ does not stabilize. Choose $a_n\in K_n-K_{n-1}$. 
Choose $\bar a_n\in F(\mathcal A)$ such $q(\bar a_n)=a_n$. Then 
$$q(\phi^{n-1}(\bar a_n))=\phi_0^{n-1}q(\bar a_n)=\phi_0^{n-1}(a_n)\ne 1$$
so $\phi^{n-1}(\bar a_n)\not\in N_0=ker(q)$ and $\bar a_n\not\in N_{n-1}$. But, 
$$q\phi^{n}(\bar a_n)=\phi_0(q(\phi^{n-1}(\bar a_n)))=\phi_0(a)=1$$
so $\phi^n(\bar a_n)\in ker(q)=N_0$ and $\bar a_n\in N_n-N_{n-1}$. 
\end{proof}

\begin{example} When $A_0$ is non-Hopfian and $\phi_0$ maps $A_0$ onto $A_0$ with non-trivial kernel, Theorem \ref{Osin} produces a corresponding ascending HNN extension with unbounded depth. 

Let  $A_0=BS(2,3)=\langle a,b:b^{-1}a^2b=a^3\rangle$, and $\phi:F(\{a,b\})\to F(\{a,b\})$ by $a\to a^2$ and $b\to b$,  observe that $\phi^i([b^{-i}ab^i,a])=[b^{-i}a^{2^i}b^i,a^{2^i}]\approx [a^{3^i},a^{2^i}]=1$, so that $[b^{-i}ab^i,a]\in N_i$. If $[b^{-i}ab^i,a]\in N_{i-1}$ then $\phi^{i-1}([b^{-i}ab^i,a])\in N_0$ where  $N_0=N(b^{-1}a^2ba^{-3})\triangleleft F(\{a,b\}))$. But 
$$\phi^{i-1}([b^{-i}ab^i,a])=[b^{-i}a^{2^{i-1}}b^{i}, a^{2^{i-1}}]\approx [b^{-1}a^{3^{i-1}}b,a^{2^{i-1}}]$$ a reduced word of syllable length $8$ in (the HNN extension) $\langle a,b:b^{-1}a^2=a^3\rangle$. In particular, the following ascending HNN extension presentation with stable letter $t$ and base group generated by $\{a,b\}$ has infinite depth:
$$\langle t,a,b:b^{-1}a^2b=a^3, t^{-1}at=a^2, t^{-1}bt=b\rangle.$$
\end{example}


Since $\phi_1$ is an isomorphism (see $(\ast)$), $\langle A\rangle=\langle a,b\rangle$ is normal in $G$ and the main theorem of M. Mihalik's paper \cite{M1} implies $G$ is semistable at $\infty$. So this particular approach cannot yield a non-semistable at $\infty$ ascending HNN extension of unbounded depth when $\phi_0$ is an epimorphism.

The remainder of this section is of general interest in understanding presentations of ascending HNN extensions, but not important to the proof of our main theorem.

\begin{remark} \label{R2} 
Consider a homomorphisms $\phi: F(\mathcal A)\to F(\mathcal A)$ for $\mathcal A$ finite where $\phi$ has non-trivial kernel. One might wonder if it is possible to have a such a homomorphism so that (even with $\mathcal R=\emptyset$), the presentation $\langle t, \mathcal A: t^{-1}at=\phi(a) \hbox{ for } a\in \mathcal A\rangle$  does not have finite depth? I.e. is it possible that the ascending collection of normal subgroups of $F(\mathcal A) $ defined by $N_k=\langle \cup _{i=1}^kker(\phi^{i})\rangle$ does not stabilize? The answer is no.

Consider the sequence $F(\mathcal A)\to \phi(F(\mathcal A))\to \phi^2(\mathcal F(\mathcal A))\to \cdots$ of epimorphisms where each map is $\phi$. For $i>0$,  $\phi^i(F(\mathcal A))$ is a free group of rank $\leq rank(\phi^{i-1}(F(\mathcal A)))$. So, for some integer $m\geq 0$, $rank (\phi^m(F(\mathcal A)))=rank (\phi^{m+1}(F(\mathcal A)))$. As finitely generated free groups are Hopfian, the epimorphism $\phi:\phi^m(F(\mathcal A))\to \phi^{m+1}(F(\mathcal A))$ is an isomorphism and $ker(\phi^m)=ker(\phi^{m+1})$.
\end{remark} 

Next we show that any homomorphism $\phi :F(\mathcal A)\to F(\mathcal A)$ defining an ascending HNN extension can be replaced by a monomorphism. 

\begin{lemma} 
Suppose $\mathcal A$ is a finite set, $\mathcal R$ is a finite subset of the free group $F(\mathcal A)$ and $\phi:F(\mathcal A)\to F(\mathcal A)$ is a homomorphism. Then there is a finite set $\mathcal B$, a finite set $\mathcal R'\subset F(\mathcal B)$,  a monomorphism $\phi':F(\mathcal B)\to F(\mathcal B)$ and an isomorphism of ascending HNN extensions:
$$\langle t,\mathcal A: \mathcal R, t^{-1}at=\phi(a)\hbox{ for }a\in \mathcal A\rangle{\buildrel \rho\over \longrightarrow } \langle t,\mathcal B: \mathcal R', t^{-1}bt=\phi'(b) \hbox{ for } b\in \mathcal B\rangle$$

Furthermore,  if 
$$q_{\mathcal A}:F(\mathcal A\cup \{t\})\to  \langle t,\mathcal A: \mathcal R, t^{-1}at=\phi(a)\hbox{ for }a\in \mathcal A\rangle\hbox{  and}$$ 
        $$q_{\mathcal B}:F(\mathcal B\cup \{t\})\to \langle t,\mathcal B: \mathcal R', t^{-1}bt=\phi'(b) \hbox{ for } b\in \mathcal B\rangle$$ 
are the natural projections, then 
there is a epimorphism 
$$\rho':F(\mathcal A\cup \{t\})\to F(\mathcal B\cup \{t\})\hbox{ such that:}$$

1) $\rho'(t)=t$

2) $\rho' \circ q_{\mathcal B}=q_{\mathcal A}\circ \rho$ and 

3) $\rho'(\mathcal R)=\mathcal R'$,  (for $N_G(\mathcal R)$ the normal closure of $\mathcal R$ in $G$) 

$\ \ \ \rho'(N_{F(\mathcal A)}(\mathcal R))= N_{F(\mathcal B)}(\mathcal R')$ and $\rho'(N_{F(\mathcal A\cup \{t\})}(\mathcal R))= N_{F(\mathcal B\cup \{t\})}(\mathcal R')$

\noindent In particular, the following diagram commutes:

$$ F(\mathcal A\cup \{t\}) {\buildrel \rho'\over \longrightarrow}  F(\mathcal B\cup \{t\})$$
$$ \downarrow q_{\mathcal A} \ \ \ \ \ \ \ \ \ \ \ \ \ \    \downarrow q_{\mathcal B}$$
$$\langle t,\mathcal A: \mathcal R, t^{-1}at=\phi(a)\rangle {\buildrel \rho\over \longrightarrow}   \langle t,\mathcal B: \mathcal R', t^{-1}bt=\phi'(b)\rangle$$
(Basically $\rho$ is conjugation by $t^m$ for some $m\geq 0$.)
\end{lemma}
\begin{proof}
Since free groups are Hopfian, there is an integer $m\geq 0$ such that $\phi:\phi^m(F(\mathcal A))\to \phi^{m+1}(F(\mathcal A))$ is an isomorphism (see Remark \ref{R2}). Let $\mathcal B$ be a finite set of free generators for $\phi^m(F(\mathcal A))$ (so $F(\mathcal B)\equiv \phi^m(F(\mathcal A))$) and let $\phi':F(\mathcal B)\to F(\mathcal B)$ be defined so that $\phi'(b)$ is a $\mathcal B$-word for $\phi(b)$ for each $b\in \mathcal B$. Note that $\phi'$  is a monomorphism, since $\phi:\phi^m(F(\mathcal A))\to\phi^{m+1}(F(\mathcal A))<F(\mathcal B)$ is a monomorphism.

Define $\rho':F(\mathcal A\cup \{t\})\to F(\mathcal B\cup \{t\})$ such that $\rho'(t)=t$ and $\rho'(a)=\phi^m(a)$ for all $a\in \mathcal A$. Note that $\rho'$ is an epimorphism. Let $\mathcal R'=\phi^m(\mathcal R)$ (written as $\mathcal B$-words) and then 3) holds. Since $\rho'$ of each relation of $\langle t,\mathcal A: \mathcal R, t^{-1}at=\phi(a)\rangle$ is a relator of $\langle t,\mathcal B: \mathcal R', t^{-1}bt=\phi'(b)\rangle$, the homomorphism  $\rho$ can be defined so that 2) holds. Since $\rho'$ is an epimorphism, $\rho$ is an epimorphism. (Basically, $\rho$ is conjugation by $t^m$.) 

    To show $\rho$ is an isomorphism, it remains to show that if $w\in ker(\rho q_{\mathcal A})$ then $w\in ker (q_{\mathcal A})$ (i.e. $\rho$ is a monomorphism). First observe that the exponent sum of $t$ in $w$ is zero. Next observe that, $w\in ker(\rho q_{\mathcal A})$ (respectively $w\in ker(q_{\mathcal A})$) iff $t^{-j}wt^j\in ker(\rho q_{\mathcal A})$ (respectively $t^{-j}wt^j\in ker( q_{\mathcal A})$) for every integer $j\geq 0$. Select a positive integer $j$ such that any initial segment of $t^{-j}wt^j$ has $t$-exponent sum $\leq 0$. In $F(\mathcal A\cup\{t\})$, $w=(t^{-n_1}w_1t^{n_1})\cdots (t^{-n_s}w_st^{n_s})$ where $n_i\geq 0$ and each $w_i\in F(\mathcal A)$. Let $\bar w \equiv \phi^{n_1}(w_1)\cdots \phi^{n_s}(w_s)(\in F(\mathcal A))$. Now, $q_{\mathcal A}(w)=q_{\mathcal A}(\bar w)$ and $\bar w\in ker(q_{\mathcal B}\rho')$. Note that $\rho'(\bar w)=\phi^m(\bar w)\in ker(q_{\mathcal B}) (<F(\mathcal B))$.
By Theorem \ref{rel}, $\phi^m(\bar w)\in (\phi')^{-k}(N(\cup_{i=0}^{\infty}(\phi')^i(\mathcal R')))$ for some integer $k\geq 0$. 
By 3) we have,  $\phi^m(\bar w)\in \phi^{-k}(N(\cup_{i=0}^{\infty}\phi^i(\phi^m(\mathcal R))))$ and so $\bar w\in  \phi^{-k-m}(N(\cup_{i=m}^{\infty}\phi^i(\mathcal R)))$. By Theorem \ref{rel}, $\bar w$ (and hence $w$) is an element of $ker (q_{\mathcal A})$. 
\end{proof}

  \section{Bounded Depth HNN extensions  are semistabile at $\infty$}\label{Smain}

The group $G$ is an ascending HNN extension of a finitely generated group $A$ and $G$ has bounded depth.   
We use the notation of $\S$\ref{basess}. Let $\mathcal A=\{a_1,\ldots , a_n\}$ be a finite generating set for $A$ and 
$$\mathcal P\equiv \langle t,\mathcal A:\mathcal R,t^{-1}at=\phi(a) \hbox { for all } a\in \mathcal A\rangle$$
a finite presentation for $G$, where each element of $\mathcal R$ is an $\mathcal A$-word. 
Let $X$ be the Cayley 2-complex for this presentation, and $\Lambda$ be the Cayley graph of $A$ with generating set $\mathcal A$. We assume $\ast\in \Lambda\subset X$ where $\ast$ is the identity vertex for $X$.  We must show condition (2) of Theorem \ref{GGM} is satisfied for each compact set $C$ in $X$. 
We will show that there is an integer $N(C)\geq 0$ (defined in Lemma \ref{below}) such that $t^NAt^{-N}$ is co-semistable at $\infty$ in $X$ with respect to $C$. This requires that we find a compact set $D(C)$ such that loops in $X-(t^NAt^{-N} ) D(C)$ can be pushed to infinity by proper homotopies in $X-C$. In every instance $D(C)$ will have the form $t^{N(C)}\{\ast, t^{-1},\ldots ,t^{-M}\}$ for some integer $M$ that depends on $C$ and the depth of the presentation $\mathcal P$ for $G$. 

\begin{remark} \label{stcoax} 
In the case that $A$ is finitely presented, it is interesting to note that our proof will show that for our choice of $D(C)$, each loop in $X-(t^{-N}At^N)D$ is homotopically trivial in $X-(t^{-N}At^N) D$ (see Theorem \ref{FP}). This sort of behavior is related to the main theorems of \cite{W92}, \cite{GGM16} and \cite{GG12}, and is called {\it strongly coaxial} when $A$ is infinite cyclic.
\end{remark} 

Recall that $P:X\to \mathbb R$ is such that for each vertex $v\in G\subset X$, $P(v)$ is the exponent sum of $t$ in $v$ and we say $v$ is in {\it level} $P(v)$.
 The next lemma is a direct consequence of the normal form for elements of $G$ (each element $g\in G$ has the form $t^nat^{-m}$ for some $n, m\geq 0$ and $a\in A$). 
 \begin{lemma} \label{below} 
Suppose $C$ is a finite subcomplex of $X$. For each vertex $v\in C$, write 
$$v=t^{n(v)}a_vt^{-m(v)}\hbox{ for }a_v\in A\hbox{ and }n(v), m(v)\geq 0, \ and$$ 
$$N(C)=max\{n(v):v\in C\}\hbox{ and }M(v,C)=N(C)-n(v)+m(v)(\geq 0).$$  Then $vt^{M(v,C)}\in t^{N(C)}A$. 

\noindent Note that by definition, $N(C)-M(v,C)=n(v)-m(v)=P(v)$.  
 \end{lemma} 
 \begin{proof}
 For $v\in C$, 
 $$v=t^{N(C)}(t^{n(v)-N(C)}a_vt^{N(C)-n(v)})t^{-M(v,C)}.$$ 
 If $a'_v=t^{n(v)-N(C)}a_vt^{N(C)-n(v)}(\in \phi^{N(C)-n(v)}(A)<A)$ then $vt^{M(v,C)}=t^{N(C)}a'$. 
 \end{proof}
 Geometrically this say that for each vertex $v$ of $C$, the edge path at $v$ with each edge labeled $t$ and of length $M(v,C)$ ends in $t^{N(C)}A$. 

\begin{lemma} \label{below2} 
Suppose $C$ is a finite subcomplex of $X$. Let  
 $$M(C)=max\{M(v,C): v\hbox{ is a vertex of } C\}.$$
Then for each vertex $v\in C$   
$$v\in (t^{N(C)}At^{-N(C)})(t^{N(C)}\{\ast, t^{-1},\ldots ,t^{-M(C)}\}) \ \hbox{and}$$ 
for positive integers $M,N$ and $w\in (t^{N}At^{-N})(t^{N}\{\ast, t^{-1},\ldots ,t^{-M}\})$ we have 
$$wA\subset (t^{N}At^{-N})(t^{N}\{\ast, t^{-1},\ldots ,t^{-M}\}).$$
\end{lemma}   
\begin{proof} The first conclusion follows from Lemma \ref{below}.  Note that $w=t^{N}at^{-m}$ for some $a\in A$ and $m\in \{0,\ldots, M\}$.  

Then $wA\subset t^{N}a(t^{-m}At^m)t^{-m}$ and as $t^{-m}At^m\subset A$:
$$wA\subset  t^{N}At^{-m}\subset (t^{N}At^{-N})(t^{N}\{\ast, t^{-1},\ldots ,t^{-M}\}).$$

\end{proof}

 
For integers $N,M\geq 0$ define $ D(N,M)\equiv t^NA\{\ast,t^{-1},\ldots, t^{-M}\}$. If $C$ is compact in $X$, and $B$ is the bounded depth of our ascending HNN presentation $\mathcal P$ we will use the set $D(N(C),M(C)+B+1)$ to play the roll of the compact set $D$ in $X$ and $t^{N(C)}At^{-N(C)}$ to play the roll of $J$ when applying Theorem \ref{GGM}. First we must understand the set $(t^NAt^{-N})D(N,M)=t^NA\{\ast, t^{-1},\ldots, t^{-M}\}$ and a few geometric definitions will help. If $v,w\in G$, we say the coset $wA$ is {\it $n$ levels directly below} $vA$ if there is an edge path of length $n$ with each edge labeled $t$ from a vertex of $wA$ to a vertex of $vA$. Note that if $wA$ is $n$ levels directly below $vA$ then for every vertex $u$ of $wA$, the edge path at $u$ of length $n$ and with each edge labeled $t$ ends in $vA$. We say $vA$ is {\it $n$ levels directly above} $wA$. Any coset $wA$ has exactly one coset $n(\geq 0)$ levels directly above it, but the cosets one level directly below $vA$ are in 1-1 correspondence with the cosets of $A$ in $G$.  This means

\begin{lemma} \label{directly} 
The set $D(N,M)=t^NA\{\ast,t^{-1},\ldots, t^{-M}\}$ is the union of cosets $vA$ that are $n$ levels directly below $t^NA$ for $n\in\{0,1,\ldots, M\}$.
\end{lemma}

\noindent {\bf Note.} In order to avoid confusion we may use the notation $H\cdot E$ instead of $HE$ when $H$ is a subgroup of $G$ and $E$ a subset of  $X$.
 
Let  $Q(M)=\{\ast, t^{-1}, t^{-2},\ldots, t^{-M}\}$ ($M\geq 0$) and notice that
the next lemma says that it is easy to check if a vertex $v$ of $X$ is in either $A \cdot Q(M)$, $K_0$ (a special component of $X-A\cdot Q(M)$) or a component of $X-A \cdot Q(M)$ other than $K_0$. If $v$ is in a level $>0$ then $v\in K_0$. If $v$ is in level $0$ through $-M$ then $v$ is in $A\cdot Q$ if the edge path from $v$ to level $0$, with each edge labeled $t$, ends in $A$ (i.e. $vA$ is $-P(v)$ levels directly below A); and $v$ is in $K_0$ otherwise. If $v$ is in a level $<-M$, then $v$ is in $K_0$ if the edge path from $v$ to level $0$, with each edge labeled $t$, does not end in $A$; and otherwise, $v$ belongs to a component of $X-A\cdot Q$ other than $K_0$. 
Note that $t^n\in K_0$ for all $n>0$, so that under the quotient of $X$ by $A$, the image of $K_0$ is not contained in a compact set. If $v\in K$ where $K$ is a  component of $X-A\cdot Q$ other than $K_0$ then $vt^n\in K$ for all $n<0$, so under the quotient of $X$ by $A$, the image of $K$ is not contained in a compact set. Our terminology for this is that $K$ and $K_0$ are $A$-unbounded components of $X-A\cdot Q$.

\begin{lemma} \label{nbhd} 
Let $Q(M)=\{\ast, t^{-1}, t^{-2},\ldots, t^{-M}\}$ for $M\geq 0$. Then 

1) $A\cdot Q(M)$ is the set of all vertices $v\in X$ such that $P(v)\in \{-M,\ldots, 0\}$ and $vt^{-P(v)}\in A$. Furthermore, if $v\in A\cdot Q(M)$ then $vA\subset A\cdot Q(M)$.

2) $X-A\cdot Q(M)$ has an $A$-unbounded component $K_0$ with stabilizer $A$ and the vertex $v$ of $X-A\cdot Q(M)$ is in $K_0$ if and only if either $P(v)\geq -M$ or both $P(v)<-M$ and $vt^{-P(v)}\not\in A$, 

3) if $K$ is any component of $X-A\cdot Q(M)$ other than $K_0$, then $K$ is $A$-unbounded, and if $v$ is a vertex of $K$, then $P(v)<-M$ and $vt^{-P(v)}\in A$.
\end{lemma}

\begin{proof} Part 1): This part follows directly from Lemma \ref{directly} (with $N=0$). 

Part 2): Let $K_0$ be the component of $X-A\cdot Q$ that contains the vertex $t$. Let $v$ be a vertex of $X$, then by normal forms, $v=t^lat^{-m}$ where $a\in A$ and $l,m\geq 0$. If $P(v)>0$, then $l>m$ and the normal form for $v$ defines an edge path from $t$ to $v$ in levels 1 and above, and hence avoiding $A\cdot Q$. So if $P(v)>0$, then $v\in K_0$. Note that  $P(at)=1$ for all $a\in A$, so that $A$ stabilizes $K_0$. 

Suppose $v\in X-A\cdot Q$ and  $P(v)\in \{-M,\ldots, 0\}$, then by part 1), $vt^{-P(v)}\not\in A$ and no point of the edge path beginning at $v$ with labeling $t^{-P(v)}$ is a point of $A\cdot Q$. Since $P(vt^{-P(v)+1})=1$, the edge path at $v$ with labeling $t^{-P(v)+1}$ avoids $A\cdot Q$ and ends at a point of $K_0$. So if $v\in X-A\cdot Q$ and $P(v)\in \{-M,\ldots,0\}$ then $v\in K_0$. 

Suppose $v\in X-A\cdot Q$ and $P(v)<-M$. Note that $P(vt^{-P(v)})=0$. If $vt^{-P(v)}\not \in A$, then we have already shown that $vt^{-P(v)}\in K_0$, and by part 1), no point of the path with labeling $t^{-P(v)}$ at $v$ intersects $A\cdot Q$. Hence $v\in K_0$. 

For the converse, suppose $v\in K_0$ and $P(v)<-M$. We must show $vt^{-P(v)}\not\in A$. Let $\alpha$ be an edge path in $X-A\cdot Q$ from $t$ to $v$. Let $\beta$ be a tail of $\alpha$ where $w$, the initial point of $\beta$, is the last point of $\alpha$ with $P(w)=-M$. 
The first edge of $\beta$ is labeled $t^{-1}$. Note that conjugation relations allow us to move each $A$-edge of $\beta$ up to level $-M$ so there is an edge path from $w$ to $v$ labeled $ (x_1,\ldots x_i, t^{-k})$ where $k> 0$ and $x_i\in \{a_1,\ldots, a_n\}^{\pm 1}$. Hence $vt^k\in wA$ and $P(vt^k)=-M$. By Part 1), $w\not\in A\cdot Q$ implies $wA \cap A\cdot Q=\emptyset$, so $vt^k\not\in A\cdot Q$. Again by Part 1), $vt^kt^{-P(vt^k)}\not\in A$. Then $vt^{-P(v)}=vt^kt^{-P(v)-k}=vt^kt^{-P(vt^k)}\not\in A$. This completes part 2).

Part 3): If $v\in K\ne K_0$ then by Part 2), $P(v)<-M$ and $vt^{-P(v)}\in A$.  
\end{proof}

We need a slightly stronger version of Lemma \ref{nbhd}.  Recall that $Q(M)=\{\ast, t^{-1}, t^{-2},\ldots, t^{-M}\}$. Then 
$$t^NA\cdot Q(M)=t^NAt^{-N} (t^N(Q(M))).$$
Observe that for any integer $m\geq 0$ the stabilizer of $t^m\Lambda$ is $t^mAt^{-m}$.
\begin{lemma} \label{nbhd2} 
Let  $M,N\geq 0$ be integers:

1) The set $t^NA\cdot Q(M)(=D(N,M))$  consists of the vertices $v\in X$ such that $P(v)\in \{N,N-1,\ldots,N-M\}$ and $vt^{N-P(v)}\in t^NA$. Furthermore, if $v\in t^NA\cdot Q(M)$ then $vA\subset t^NA\cdot Q(M)$.

2) Let $K_0$ be the component of $X-A\cdot Q(N)$ described by part 2 of Lemma \ref{nbhd}. Then $t^NK_0$ is a $(t^NAt^{-N})$-unbounded component of $X-t^NA\cdot Q(N)$ with stabilizer $t^NAt^{-N}$, and the vertex $v$ of $X-t^NA\cdot Q$ is in $t^NK_0$ if and only if either $P(v)\geq N-M$ or $P(v)<N-M$ and $vt^{N-P(v)}\not\in t^NA$, 

3) if $K$ is any component of $X-A\cdot Q$ other than $K_0$ then $t^NK$ is a $(t^NAt^{-N})$-unbounded component of $X-t^NA\cdot Q(M)$, and if $v$ is a vertex of $t^NK$, then $P(v)<N-M$ and $vt^{N-P(v)}\in t^NA$. 
\end{lemma}
\begin{proof}
Part 1): If $v\in t^NA\cdot Q(M)$, then $P(v)\in \{N,N-1,\ldots, N-M\}$. Note that  $P(t^{-N}v)=-N+P(v)\in \{0,\ldots, -M\}$. Lemma \ref{nbhd} implies, $t^{-N}v\in A\cdot Q$ if and only if  $t^{-N}vt^{-P(t^{-N}v)}\in A$ if and only if $vt^{N-P(v)}\in t^NA$. Furthermore if $v\in t^NA\cdot Q(M)$ then $t^{-N}v\in A\cdot Q(M)$ and by Lemma \ref{nbhd}, $t^{-N}vA\subset A\cdot Q(M)$ so that $vA\subset t^NA\cdot Q(M)$. 

Part 2): By Lemma \ref{nbhd},  $t^NK_0$ is a component of $X-t^N(A\cdot Q(M))$. Since $t\in K_0$, $t^{N+1}\in t^NK_0$, and so the proper ray at $t^{N+1}$ with all edge labels $t$ belongs to $t^NK_0$. In particular, $t^NK_0$ is $t^NAt^{-N}$-unbounded.  Since $A$ stabilizes $A\cdot Q(M)$, $t^NAt^{-N}$ stabilizes $t^NA\cdot Q(M)$. The vertex $v$  of $X$ belongs to $t^NK_0$ if and only if $t^{-N}v\in K_0$, (by Lemma \ref{nbhd}) if and only if $P(t^{-N}v)\geq -M$ or both $P(t^{-N}v)< -M$ and $t^{-N}vt^{-P(t^{-N}v)}\not\in A$, if and only if $P(v)\geq N-M$ or both $P(v)<N-M$ and $vt^{N-P(v)}\not\in t^NA$.  

Part 3): Suppose $v$ is a vertex of $t^NK$ then $t^{-N}v\in  K$.  By Lemma \ref{nbhd}, $P(t^{-N}v)<-M$ (so $P(v)<N-M$) and $t^{-N}vt^{-P(t^{-N}v)}\in A$ (so $vt^{N-P(v)}\in t^NA$).
\end{proof}

Geometrically, the only difference between Lemma \ref{nbhd2} and Lemma \ref{nbhd} is that in order to check if a vertex $v$ in a level of $X$ less than $N$, belongs to either $t^NA\cdot Q(M)$, $t^NK_0$ or $t^NK$ for $K$ a component of $X-A\cdot Q(M)$ different than $K_0$, one simply checks if the end point of the path at $v$ with each edge labeled $t$ and ending in level $N$, ends in $t^NA$ or not. It is also important to observe the following remark.

\begin{remark}\label{coset} 
For any integers $M,N\geq 0$ the set $t^NA\cdot Q(M)(=D(N,M))$ and any component of $X-D(N,M)$ is a union of cosets $vA$. 
\end{remark} 

\begin{lemma} \label{up} 
Suppose $M,N\geq 0$ are integers and $v$ is a vertex of the  component $t^NK_0$ of $X-t^NA\cdot Q(M)$. Then for any integer $n\geq 0$, $(vt^nA)\cap A\cdot Q(M)=\emptyset$.
\end{lemma}
\begin{proof} 
By 1) of  Lemma \ref{nbhd2}, it suffices to show that $vt^n\not\in t^NA\cdot Q(M)$. But this follows directly from parts 1) and 2) of Lemma \ref{nbhd2}.
\end{proof}

\noindent $(\ast )$ From this point on we assume the presentation $\mathcal P$ has bounded depth $B\geq 0$.
 
 \begin{lemma} \label{killD} 
 If $\alpha$ is an edge path loop in $X$ and $im(P(\alpha))\subset (-\infty,L]$, then $\alpha$ is homotopically trivial by a homotopy $H$ such that $im(P(H))\subset (-\infty, L+B]$. 
 \end{lemma}
 \begin{proof}
Using only conjugation 2-cells, $\alpha$, is homotopic (by a homotopy $H_1$) to an edge path loop $\beta$, each of whose vertices is in level $L$. In particular, each edge of $\beta$ is labeled by an element of $\mathcal A$ and $im(P(H_1))\subset (-\infty, L]$. The word $w$ determined by the edge labeling of $\beta$ is in the kernel of the epimorphism $p:F(\mathcal A)\to A$. So $w\in\cup_{i=0}^{B}\phi^{-i}(N_0(\mathcal R,\phi))$.  By Remark \ref{loopkill}, the loop $\beta$ (and hence $\alpha$) is homotopically trivial by a homotopy $H$ such that $im(P(H))\subset (-\infty, L+B]$. 
\end{proof}

\begin{lemma} \label{hup} 
Suppose $M, N\geq 0$ are integers, $\alpha$ is a loop in $X-t^NA\cdot Q(M)$ and $B$ is the bounded depth of the presentation $P$. 

1) If $\alpha$ has image in a component of $X-t^NA\cdot Q(M)$ other than $t^NK_0$, then $\alpha$ is homotopically trivial by a homotopy $H$ such that $P(H)$ has image in $(-\infty, B+N-M]$,

2) if $\alpha$ has image in $t^NK_0$, $v$ is a vertex of $\alpha$ and $r_v$ is the proper edge path ray at $v$ with each edge labeled $t$, then there is proper homotopy $H:[0,\infty)\times [0,1]\to t^NK_0$ where $H(x,0)=H(x,1)=r_v(x)$, and $H(0,y)=\alpha (y)$. 
\end{lemma} 

\begin{proof} 
Part 1): By 3) of Lemma \ref{nbhd2}, $im(P(\alpha))\subset (-\infty, N-M]$. Lemma \ref{killD} finishes part 1). 

Part 2): Let $H$ be the homotopy that strings together the homotopies $H_e$ of Lemma \ref{push} for each $\mathcal A$-edge $e$ of $\alpha$. The image of $H$ avoids $t^NA\cdot Q(M)$ by Lemma \ref{up} and so is in $t^NK_0$. The homotopy $H$ is proper since it is a combination of finitely many proper homotopies.
\end{proof}

By Lemma \ref{below2}, if $C$ is a compact subset of $X$, there are integers $M(C)$ and $N(C)$ such that $C\subset t^NA\cdot Q(M)$. 
\begin{theorem} \label{FP} 
Suppose $G$ is an ascending HNN extension of the finitely presented group $A$ and $X$ is the Cayley 2-complex for the HNN presentation with stable letter $t$ and base $A$ (with a finite presentation of $A$ as a sub-presentation). If $M, N\geq 0$ are integers and $\alpha$ is a loop in $X-t^NA\cdot Q(M)$ $(=X-t^NAt^{-N}\cdot (t^NQ(M)))$ then $\alpha$ is homotopically trivial in $X-t^NA\cdot Q(M)$. 
\end{theorem}
\begin{proof}
We present the case where $N=0$ as all others are completely analogous.
Let $\Lambda$ be the Cayley 2-complex for $A$, determined by the presentation of $A$ within our HNN presentation of $G$.  If $K$ is a component of $X-A\cdot Q$ other than $K_0$ and $\alpha$ is an edge path loop in $K$, then each vertex $v$ of $\alpha$ is such that $P(v)<-M$. Using conjugation relations $\alpha$ is homotopic in $K$ to an $A$-loop $\alpha_1$ in level $-M-1$. Then $\alpha_1$ lies in a copy of $\Lambda$ in level $-M-1$ and so is homotopically trivial in level $-M-1$.

If $\alpha$ is an edge path loop in $K_0$, then by Lemma \ref{up}, conjugation relations can be used to show that $\alpha$ is homotopic to a loop $\alpha_1$ in a single level and this homotopy avoids $A\cdot Q$. Lemma \ref{up} also implies that $\alpha_1$ is in a copy of $\Lambda$ that avoids $A\cdot Q$. As $\alpha_1$ is homotopically trivial in that copy of $\Lambda$, $\alpha_1$ (and hence $\alpha$) is homotopically trivial in $X-A\cdot Q$.
\end{proof}

Suppose $M,N\geq 0$ are integers and $s$ is a proper edge path ray in $X-t^NA\cdot Q(M)$ with initial vertex $v\in K_0$. If $q$ is the quotient of $X$ by the action of $t^NAt^{-N}$ and $qs$ has image in a compact subset of $(t^NAt^{-M})\backslash X$ (so $s$ is $t^NAt^{-N}$-bounded), then each vertex of $s$ is within edge path distance $\leq K$ of $t^NA$ and $Ps$ has image in the  closed interval $[N-K, N+K]$. 

\begin{lemma} \label{corner} 
Suppose $M,N\geq 0$ are integers, $s$ is a proper edge path ray in the $t^NK_0$ component of $X-t^NA\cdot Q(M)$ and $s(0) =v$. Let $r_v$ be the proper edge path ray at $v$, each of whose edges is labeled $t$.  If $Ps$ has image in a closed interval then $s$ is properly homotopic to $r_v$ by a homotopy with image in $t^NK_0$. 
\end{lemma}
\begin{proof}
Assume that the image of $Ps$ is $[L,M]$. By Lemma \ref{up},  one can use conjugation relations to slide each $A$-edge of $s$ along $t$-edges to level $M$, by a homotopy with image in $t^NK_0$. So $s$ is properly homotopic to $s'$, the resulting proper ray which (after removing any backtracking edges $(t,t^{-1})$ or $(t^{-1},t)$) is a proper $\mathcal A$-ray. Let $r'$ be the proper edge path ray at the initial point of $s'$ with all edges labeled $t$ (so $r'$ is a sub-ray of $r_v$).  Let $H$ be the proper homotopy of $s'$ to $r'$ defined in Lemma \ref{string}. By Lemma \ref{up}, $H$ has image in $t^NK_0$.\end{proof}

\begin{proof} {\bf (of Theorem \ref{mainbd})} Let $X$ be the Cayley 2-complex of $\mathcal P$. By Proposition \ref{strongss}, $t^NAt^{-N}$ is  semistable at $\infty$ in $X$ for all $N\geq 0$ and in $\S$\ref{ss} we reduced the proof of Theorem \ref{mainbd} to showing that for each compact set $C$ in $X$ there is an integer $N\geq 0$ such that $t^NAt^{-N}$ is co-semistable at $\infty$ in $X$ with respect to $C$. That means:

For any finite subcomplex $C$ of $X$ there is an integer $N\geq 0$, and compact set $D$ such that for any proper $t^NAt^{-N}$-bounded ray $s$ in $X-t^NAt^{-N}D$ and  loop $\alpha$ in $X-t^NAt^{-N}D$ such that $\alpha(0)=s(0)$, there is a proper homotopy $H:[0,1]\times [0,\infty)\to X-C$ such that $H(0,t)=H(1,t)=s(t)$ and $H(t,0)=\alpha$.    

Start with a finite subcomplex $C$ of $X$.  The integer $N(C)\geq 0$ will play the part of $N$. Recall that $B$ is the bounded depth of the presentation $\mathcal P$. Let  
$$D=t^{N(C)} Q(M(C)+B+1)$$ 
Recall $Q(M)=\{\ast, t^{-1},\ldots ,t^{-M}\}$. By Lemma \ref{below2}, for each vertex $v\in C$:
$$vA\subset  t^{N(C)}At^{-N(C)}(t^{N(C)}Q(M(C)))$$ 
$$\subset t^{N(C)}At^{-N(C)}(t^{N(C)}Q(M(C)+B+1))=$$
$$t^{N(C)}At^{-N(C)}D=t^{N(C)}A\cdot Q(M(C)+B+1).$$
If $v\in C$, then $v\in t^{N(C)}A Q(M(C))$ so that $P(v)\in [N(C)-M(C),N(C)]$.  
Suppose $\alpha$ is a loop in $X-t^{N(C)}At^{-N(C)}D$. Then $\alpha$ is either in $t^{N(C)} K_0$ where $K_0$ is the special component of $X-A\cdot Q(M(C)+B+1)$ (described in part 2) of Lemma \ref{nbhd}) or $\alpha$ is in $t^{N(C)}K$ for some  component $K$ of $X-A\cdot Q(M(C)+B+1)$ other than $K_0$. If $\alpha$ belongs to $t^{N(C)}K$, then by part 1) of Lemma \ref{hup}, $\alpha$ is homotopically trivial by a homotopy $H$ such that 
$$im(P(H))\subset (-\infty, B+N(C)-(M(C)+B+1)]=(-\infty, N(C)-M(C)-1].$$ 
Since $P(C)\subset [N(C)-M(C),N(C)]$, the homotopy $H$ kills $\alpha$ in $X-C$ (actually in $X-A\cdot C$).

If $\alpha$ is in $t^{N(C)}K_0$, and $s$ is a $t^{N(C)}Dt^{-N(C)}$-bounded proper ray in $t^{N(C)}K_0$ such that $\alpha(0)=s(0)$, then by Lemma \ref{corner}, $s$ is properly homotopic (rel$\{s(0)\}$) to $r$  the proper edge path ray at $s(0)$, each of whose edges is labeled $t$, by a homotopy with image in $t^{N(C)}K_0\subset X-C$. Combining the homotopy of $r$ to $s$ with one given by part 2) of Lemma \ref{hup} (also in $t^{N(C)}K_0$) completes the proof.
\end{proof} 

\bibliographystyle{amsalpha}
\bibliography{paper}{}

Michael Mihalik

Department of Mathematics, Vanderbilt University, Nashville, TN 37240

email: michael.l.mihalik@vanderbilt.edu
  
 \end{document}